\documentclass[11pt]{article}
\date{}

\usepackage{graphicx}
\usepackage{tabularx}
\usepackage{url}
\usepackage{amsthm}
\usepackage{amsmath}
\usepackage{amsfonts}
\usepackage{subcaption}

\usepackage{color}
\usepackage{xspace}
\usepackage{tikz}
\usetikzlibrary{decorations,decorations.pathmorphing,decorations.pathreplacing,fit,positioning,arrows}

\theoremstyle{plain}
\newtheorem{theorem}{Theorem}
\newtheorem{lemma}{Lemma}

\newtheorem{claim}{Claim}
\newtheorem{fact}{Fact}

\newtheorem{corollary}{Corollary}
\theoremstyle{definition}
\newtheorem{definition}{Definition}

\theoremstyle{remark}

\usepackage{graphicx}
\usepackage{tabularx}
\usepackage{url}
\usepackage{amsmath}
\usepackage{amsfonts}
\usepackage{amssymb}
\usepackage{tikz}
\usepackage{hhline}
\usepackage{array}

\usetikzlibrary{decorations.pathreplacing}
\usetikzlibrary{shapes}
\usetikzlibrary{shapes.multipart}
\tikzstyle{vertex}=[circle,fill=black!100,text=white,inner sep=0.8mm]
\tikzstyle{point}=[circle,fill=black,inner sep=0.1mm]

\DeclareMathOperator{\lett}{let}

\begin{document}
%%%%%%%%%%%%%%%%%%%%%%%%%%%%%%%%%%%%%%%%%%%%

%%%%%%%%%%%%%%%%%%%%%%%%%%%%%%%%%%%%%%%%%%%%
\title{Combinatorics and algorithms \\ for quasi-chain graphs}
\author{Bogdan Alecu\thanks{Mathematics Institute, University of Warwick, Coventry, CV4 7AL, UK. B.Alecu@warwick.ac.uk} 
\and Aistis Atminas\thanks{Department of Mathematical Sciences, Xi'an Jiaotong-Liverpool University, 111 Ren'ai Road, Suzhou 215123, China. Email: Aistis.Atminas@xjtlu.edu.cn} \and 
Vadim Lozin\thanks{Mathematics Institute, University of Warwick, Coventry, CV4 7AL, UK. V.Lozin@warwick.ac.uk} \and 
Dmitriy Malyshev\thanks{Laboratory of Algorithms and Technologies for Networks Analysis, National Research University
Higher School of Economics, 136 Rodionova Str., 603093, Nizhny Novgorod, Russia. Email: dsmalyshev@rambler.ru}$~^,$\thanks{The work of Malyshev D.S. 
was conducted within the framework of the Basic Research Program at the National Research University Higher School of Economics (HSE).}
}

\maketitle

\begin{abstract}
The class of quasi-chain graphs is an extension of the well-studied class of chain graphs.   
This latter class enjoys many nice and important properties, such as bounded clique-width, implicit representation, well-quasi-ordering by induced subgraphs, etc.
The class of quasi-chain graphs is substantially more complex. In particular, this class is not well-quasi-ordered by induced subgraphs,
and the clique-width is not bounded in it. In the present paper, we show that the universe of quasi-chain graphs is at least as complex
as the universe of permutations by establishing a bijection between the class of all permutations and a subclass of quasi-chain graphs.
This implies, in particular, that the induced subgraph isomorphism problem is NP-complete for quasi-chain graphs. On the other hand, 
we propose a decomposition theorem for quasi-chain graphs that implies an implicit representation for graphs in this class and 
efficient solutions for some algorithmic problems that are generally intractable. 
\end{abstract}

{\it Keywords}: bipartite graphs; implicit representation;  polynomial-time algorithm

%%%%%%%%%%%%%%%%%%%%%%%%%%%
\section{Introduction}
%%%%%%%%%%%%%%%%%%%%%%%%%%%

A bipartite graph is a {\it chain graph} if the neighbourhoods of the vertices in each part of its bipartition form a chain
with respect to the inclusion relation. The class of chain graphs
appeared in the literature under various names such as difference graphs \cite{difference} or half-graphs  \cite{half-graph}.
In model theory, half-graphs appear as an instance of the order property \cite{model}. The class of chain graphs is closely related to 
one more well-studied class of graphs, known as threshold graphs, and together they share many nice and important properties.
In particular, 
\begin{itemize}
\item chain graphs have bounded clique-width (and even linear clique-with), which implies polynomial-time solutions 
for a variety of algorithmic problems that are generally NP-hard;
\item chain graphs are well- (and even better-) quasi-ordered under induced subgraphs. This is because another important parameter, graph lettericity, 
is bounded for chain graphs \cite{lettericity};
\item chain graphs admit an implicit representation, which in turn implies a small induced-universal graph for the class. More specifically, 
there is a chain graph with $2n$ vertices containing all $n$-vertex chain graphs as induced subgraphs \cite{universal}.
\end{itemize}

In the terminology of forbidden induced subgraphs, the class of chain graphs is precisely the class of $2P_2$-free bipartite graphs,
i.e., bipartite graphs that do not contain the disjoint union of two copies of $P_2$ as an induced subgraph ($P_n$ denotes the chordless path on $n$ vertices).

In the present paper, we study a class of bipartite graphs that forms an extension of chain graphs defined by 
relaxing the chain property of the neighbourhoods in the following way. We say that 
a linear ordering $(a_1,\ldots,a_\ell)$ of vertices is {\it good} if for all $i<j$, the neighbourhood of $a_j$ contains at most $1$ non-neighbour of $a_i$.
We call a bipartite graph $G$ a {\it quasi-chain} graph
if the vertices in each part of its bipartition admit a good ordering. Alternatively, quasi-chain graphs  are bipartite graphs
that do not contain an unbalanced induced copy of $2P_3$. Notice that $2P_3$ admits two bipartitions: one with parts of equal size (balanced)
and the other with parts of different sizes (unbalanced). In the unbalanced bipartition, one of the parts does not admit a good ordering and hence quasi-chain 
graphs are free of unbalanced $2P_3$. On the other hand, if a bipartite graph $G$ does not contain an unbalanced induced copy of $2P_3$, then
by ordering the vertices in each part in a non-increasing order of their degrees we obtain a good ordering, i.e., $G$ is a quasi-chain graph.

The class of quasi-chain graphs is substantially richer and more complex than the class of chain graphs. In particular, 
it is not well-quasi-ordered by induced subgraphs \cite{wqo} and the clique-width is not bounded in this class \cite{cw}. 
To emphasize the complex nature of this class, in Section~\ref{sec:permutations} we establish a bijection $f$ between 
the class of all permutations and a subclass of quasi-chain graphs such that a permutation $\pi$ contains a permutation $\rho$ as a pattern 
if and only if the graph $f(\pi)$ contains the graph $f(\rho)$ as an induced subgraph. Together with the NP-completeness of the {\sc pattern matching} problem for permutations
this implies the NP-completeness of the {\sc induced subgraph isomorphism} problem for quasi-chain graphs. 
  
The relationship between permutations and quasi-chain graphs also implies the existence of infinite antichains of quasi-chain graphs with respect to the induced subgraph relation 
and hence the unboundedness of lettericity in this class. In Section~\ref{sec:let}, we identify the exact boundary separating hereditary subclasses of quasi-chain graphs with bounded 
lettericity from those where this parameter is unbounded.  

In spite of the more complex structure, the quasi-chain graphs inherit some attractive properties of chain graphs.
To show this, in Section~\ref{sec:struct} we propose a structural characterisation that describes any quasi-chain graph  as the symmetric difference 
of two graphs $Z$ and $H$, where $Z$ is a chain graph and $H$ is a graph of vertex degree at most $2$. This characterisation
allows us to prove that quasi-chain graphs admit an implicit representation (Section~\ref{sec:implicit}) and that some algorithmic problems that are NP-complete for
general bipartite graphs admit polynomial-time solutions when restricted to quasi-chain graphs (Section~\ref{sec:opt}). 
All preliminary information related to the topic of the paper can be found in Section~\ref{sec:pre}.
 
%%%%%%%%%%%%%%%%%%%%%%%%%%%%%%%%%%%%%%%%%%%%%%%%%%%%%%%%%%%%%%%%%%%%%%%%%%%%%%%%%%%%%%%%%%%%%%%%%%%%%%%%%%%%%
%%%%%%%%%%%%%%%%%%%%%%%%%%%%%%%%%%%%%%%%%%%%%%%%%%%%%%%%%%%%%%%%%%%%%%%%%%%%%%%%%%%%%%%%%%%%%%%%%%%%%%%%%%%%%

\section{Preliminaries} 
\label{sec:pre}

%%%%%%%%%%%%%%%%%%%%%%%%%%%%%%%%%%%%%%%%%%%%%%%%%%%%%%%%%%%%%%%%%%%%%%%%%%%%%%%%%%%%%%%%%%%%%%%%%%%%%%%%%%%
%%%%%%%%%%%%%%%%%%%%%%%%%%%%%%%%%%%%%%%%%%%%%%%%%%%%%%%%%%%%%%%%%%%%%%%%%%%%%%%%%%%%%%%%%%%%%%%%%%%%%%%%%%%
All graphs in this paper are simple, i.e., undirected, with neither loops nor multiple edges. 
The vertex set and the edge set of a graph $G$ are denoted $V(G)$ and $E(G)$, respectively.
The {\it neighbourhood} of a vertex $v\in V(G)$ is the set of vertices adjacent to $v$. We denote the neighbourhood of $v$ 
in the graph $G$ by $N_G(v)$ and omit the subscript if it is clear from the context.

In a graph, an {\it independent set} is a subset of pairwise non-adjacent vertices and a {\it clique} is a subset of pairwise adjacent vertices.
A graph is bipartite if its vertex set can be partitioned into two independent sets, which we refer to as the parts or colour classes of the graph.
A bipartite graph $G=(V,E)$ given together with a bipartition $V=A\cup B$ is denoted $G=(A,B,E)$. 
Once such a bipartition has been fixed, we may define the {\em bipartite complement} $\widetilde{G}=(A,B,E')$ of $G$, 
in which two vertices $a\in A$ and $b\in B$ are adjacent if and only if they are not adjacent in $G$ (that is, $E' = (A \times B) - E)$. 
% Furthermore, given two bipartite graphs $G_1 = (A_1, B_1, E_1)$ and $G_2 = (A_2, B_2, E_2)$, 
% we define the {\em skew-join} of $G_1$ with $G_2$ as the graph $(A_1 \cup A_2, B_1 \cup B_2, E_1 \cup E_2 \cup A_1 \times B_2)$.

As usual, $P_n$ denotes a chordless path with $n$ vertices and $K_{p.q}$ denotes a complete bipartite graph with parts of size $p$ and $q$.
The disjoint union of $n$ copies of $G$ is denoted $nG$.

The subgraph of $G$ induced by a set $U\subseteq V(G)$ is denoted $G[U]$. If $G$ contains no induced subgraphs isomorphic to a graph $H$, then we say that $G$ is $H$-free
and call $H$ a forbidden induced subgraph for $G$. A class of graphs is {\it hereditary} if it is closed under taking induced subgraphs. It is well-known that a class 
is hereditary if and only if it can be characterised by means of minimal forbidden induced subgraphs. 

Of particular interest in this paper is the class of chain graphs. By definition, a bipartite graph $G=(A,B,E)$ is a chain graph if the vertices 
in each part can be ordered $A=(a_1,\ldots,a_\ell)$ and $B=(b_1,\ldots,b_k)$ so that $N(a_1)\supseteq\ldots\supseteq N(a_\ell)$ and $N(b_1)\subseteq\ldots\subseteq N(b_k)$.
We call this ordering {\it perfect}. A typical example of a chain graph is represented in Figure~\ref{fig-chain}. We denote a graph of this form with $n$ vertices in each part by $Z_n$. The graph $Z_n$ is
typical in the sense that it contains every chain graph with $n$ vertices as an induced subgraph, i.e., $Z_n$ is an $n$-universal chain graph \cite{universal}.   
 
%\begin{figure}[ht]
%\begin{center}
%\begin{picture}(300,60)
%\put(50,0){\circle*{5}}
%\put(100,0){\circle*{5}}
%\put(150,0){\circle*{5}}
%\put(200,0){\circle*{5}}
%\put(250,0){\circle*{5}}
%
%\put(50,50){\circle*{5}}
%\put(100,50){\circle*{5}}
%\put(150,50){\circle*{5}}
%\put(200,50){\circle*{5}}
%\put(250,50){\circle*{5}}
%
%\put(50,0){\line(0,1){50}}
%\put(250,0){\line(-1,1){50}}
%\put(250,0){\line(-2,1){100}}
%\put(250,0){\line(-3,1){150}}
%\put(250,0){\line(-4,1){200}}
%\put(100,0){\line(0,1){50}}
%\put(200,0){\line(-1,1){50}}
%\put(200,0){\line(-2,1){100}}
%\put(200,0){\line(-3,1){150}}
%\put(150,0){\line(0,1){50}}
%\put(150,0){\line(-1,1){50}}
%\put(150,0){\line(-2,1){100}}
%\put(200,0){\line(0,1){50}}
%\put(100,0){\line(-1,1){50}}
%\put(250,0){\line(0,1){50}}
%\end{picture}
%\end{center}
%\caption{A chain graph $Z_5$}
%\label{fig:chain}
%\end{figure}
 
Chain graphs are precisely $2P_2$-free bipartite graphs, i.e., $2P_2$ is the only minimal bipartite graph which is not a chain graph. 
This implies, in particular, that $G=(A,B,E)$ is a chain graph if the vertices in {\it one of the parts} can be ordered under inclusion of their neighbourhoods,
because two vertices with incomparable neighbourhoods in one part give rise to two vertices with incomparable neighbourhoods in the other part.

In this paper, we consider an extension of the class of chain graphs which can be described by forbidding an unbalanced induced copy of $2P_3$ (see Figure~\ref{fig:2P3}).
We call these graphs {\it quasi-chain graphs}, and refer to them as $2P_3$-free bipartite graphs without specifying that we forbid only an unbalanced copy of $2P_3$.  
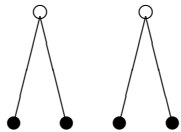
\begin{figure}[ht]
\begin{center} 
\begin{picture}(100,50)
%\put(0,0){\circle*{5}}
\put(10,0){\circle*{5}}
\put(30,0){\circle*{5}}
% \put(40,0){\circle*{5}}
\put(20,42){\circle{5}}
%\put(20,40){\line(-1,-2){20}}
\put(20,40){\line(-1,-4){10}}
%\put(20,40){\line(1,-2){20}}
\put(20,40){\line(1,-4){10}}

\put(50,0){\circle*{5}}
\put(70,0){\circle*{5}}
%\put(80,0){\circle*{5}}
\put(60,42){\circle{5}}
%\put(60,40){\line(-1,-2){20}}
\put(60,40){\line(-1,-4){10}}
%\put(60,40){\line(1,-2){20}}
\put(60,40){\line(1,-4){10}}

\end{picture} 
\end{center}
\caption{An unbalanced $2P_3$ }
\label{fig:2P3}
\end{figure}

The name {\it quasi-chain} reflects the fact that the neighbourhoods of vertices in each part create ``nearly'' a chain. 
More formally, a linear ordering $(a_1,\ldots,a_\ell)$ of vertices is {\it good} if $|N(a_j)-N(a_i)|\le 1$ for all $j>i$.
Then a bipartite graph is a quasi-chain graph if and only if the vertices in each part of its bipartition admit a good ordering.

Quasi-chain graphs appeared in the literature, without this name, in various contexts. In particular,  \cite{Allen} studies the number of $n$-vertex labelled graphs in this class,
\cite{cw} proves that the clique-width of quasi-chain graphs is unbounded, while \cite{wqo} shows that graphs in this class are not well-quasi-ordered by induced subgraphs
by establishing an intriguing relation between quasi-chain graphs and permutations. In the next section, we elaborate on this topic and show that, with some reservation, 
this relation can be developed into a bijection.

%%%%%%%%%%%%%%%%%%%%%%%%%%%%%%%%%%%%%%%%%%%%%%%%%%%%%%%%%%%%%%%%%%%%%%%%%%%%%%%%%%%%%%%%%%%%%%%%%%%%%%%%%%%%%
%%%%%%%%%%%%%%%%%%%%%%%%%%%%%%%%%%%%%%%%%%%%%%%%%%%%%%%%%%%%%%%%%%%%%%%%%%%%%%%%%%%%%%%%%%%%%%%%%%%%%%%%%%%%%

\section{Quasi-chain graphs and permutations} 
\label{sec:permutations}
%%%%%%%%%%%%%%%%%%%%%%%%%%%%%%%%%%%%%%%%%%%%%%%%%%%%%%%%%%%%%%%%%%%%%%%%%%%%%%%%%%%%%%%%%%%%%%%%%%%%%%%%%%%
%%%%%%%%%%%%%%%%%%%%%%%%%%%%%%%%%%%%%%%%%%%%%%%%%%%%%%%%%%%%%%%%%%%%%%%%%%%%%%%%%%%%%%%%%%%%%%%%%%%%%%%%%%%
Given two permutations $\pi=(\pi(1), \pi(2), \ldots, \pi(k))$ and  $\rho=(\rho(1), \rho(2), \ldots, \rho(n))$, we will write   $\pi\subseteq \rho$ 
to indicate that $\pi$ is contained in $\rho$ as a pattern, i.e., there is an order-preserving injection $e: \{1,2,\ldots,k\}\to \{1,2,\ldots,n\}$ such that 
$\pi(i)<\pi(j)$ if and only if $\rho(e(i))< \rho(e(j))$ for all $1 \leq i < j \leq k$. The pattern containment relation on permutations is the subject of a vast literature,
see, e.g., the book \cite{Kitaev} and the references therein. By mapping each permutation to its permutation graph, we transform the pattern containment relation on permutations 
into the induced subgraph relation on graphs. This mapping, however, is not injective, as it can map different permutations to the same (up to an isomorphism) graph. 
In the present section, we propose an alternative mapping from permutations to graphs: we map permutations to quasi-chain graphs, in such a way that two permutations are 
comparable if and only if their images are comparable. To make this mapping injective, we require the quasi-chain graphs to be coloured. 
That is, we will assume that every quasi-chain graph is given together with a partition of its vertex set into an independent set $A$ of white vertices
and an independent set $B$ of black vertices and we will write $G\subseteq H$ to indicate that $G$ is a coloured induced subgraph of $H$, i.e., there is 
an induced subgraph embedding of $G$ into $H$ that respects the colours. The distinction between coloured and uncoloured graphs matters, for instance, in the assignment problem.

We denote our mapping from permutations to graphs by $f$ and define it as follows. If 
$\pi=(\pi(1), \pi(2), \ldots, \pi(n))$ is an $n$-entry permutation, then $f(\pi)$ is a bipartite graph with 
parts $A=\{a_1, a_2, \ldots, a_{2n}\}$ and $B=\{b_1, b_2, \ldots, b_{2n}\}$ and the following edges:
\begin{itemize}
\item[(i)] for any $1\leq i \leq j \leq 2n$, we have $a_ib_j \in E(G)$, 
\item[(ii)] for any $1\leq i \leq n$, we have $a_{n+i} b_{\pi(i)} \in E(G)$. 
\end{itemize}
We write $G_\pi:=f(\pi)$ and say that $G_\pi$  is the quasi-permutation graph of $\pi$. 
Any graph $G$ isomorphic to $G_\pi$ for some $\pi$ will be called a quasi-permutation graph. 
It follows easily from the definition that $f$ is order-preserving, in that $\pi \subseteq \rho$ implies $f(\pi) \subseteq f(\rho)$. 

\begin{claim}
Any quasi-permutation graph $G$ is a quasi-chain graph. 
\end{claim}
\begin{proof}
We observe that the edges of type (i) define a chain subgraph of $G$ in which $N(a_j)\subseteq  N(a_i)$  for all $1 \leq i < j \leq 2n$.
The edges of type (ii) form a matching and therefore in the graph $G$ we have $|N(a_j) - N(a_i)| \leq 1$ for all $1 \leq i < j \leq 2n$.
Similarly,  $|N(b_i) - N(b_j)| \leq 1$ for all $1 \leq i < j \leq 2n$ in $G$. This shows that $A$ and $B$ have good orderings, and so any quasi-permutation graph $G$ is a quasi-chain graph.  
\end{proof}

\begin{claim}
The mapping $f$ is a bijection from the class of all permutations to the (non-hereditary) class of quasi-permutation graphs. 
\end{claim}
\begin{proof}
The mapping $f$ is surjective by the definition of quasi-permutation graphs. 
Now notice that in the graph $f(\pi)$ the degree sequence of vertices in both $A$ and $B$ is $(2,3,4, \ldots, n+1, n+1, n+2, \ldots, 2n)$. 
In particular, $f(\pi)$ uniquely determines the size of $\pi$.

The unique vertex of $A$ with degree $2$ is adjacent to vertices $b_{2n}$ and $b_{\pi(n)}$ in part $B$.
Vertex $b_{2n}$ has degree $2n$ and vertex $b_{\pi(n)}$ has degree  $k$, for some $k \le n + 1$. 
Inspecting the value of $k$ allows us to determine the value of $\pi(n)$, which is $k-1$. 
Similarly, the unique vertex of degree $3$ has three neighbours: $b_{2n}, b_{2n-1}$ and $b_{\pi(n-1)}$, which allows us to determine the value of $\pi (n-1)$. 
In this way, we see that $f(\pi)$ uniquely determines $\pi(i)$ for all $2 \leq i \leq n$. 
But two permutations with the same number of elements cannot disagree in exactly one entry, hence the graph $f(\pi)$ uniquely determines the permutation $\pi$. Therefore, $f$ is injective.   
\end{proof}

\begin{claim}\label{claim:3}
Let $\pi$ and $\rho$ be two permutations with $n$ and $m$ entries, respectively, with $n \leq m$ and $\pi(1) \neq n$. 
If $f(\pi) \subseteq f(\rho)$, then $\pi \subseteq \rho$.  
\end{claim}

\begin{proof}
Assume $f(\pi) \subseteq f(\rho)$. We denote the vertices of $f(\rho)$ as $A=\{a_1, a_2, \ldots, a_{2m}\}$ and 
$B=\{b_1, b_2, \ldots, b_{2m}\}$ and edges $a_ib_j$ if either $1 \leq i \leq j \leq 2m$ or $m+1 \leq i \leq 2m$ and   $j=\rho (i-m)$.
Also, we denote the vertices of $f(\pi)$ as $A'=(a_1', a_2', \ldots, a_{2n}')$, and $B'=(b_1', b_2', \ldots, b_{2n}')$ 
with edges $a_i'b_j'$ if either $1 \leq i \leq j \leq 2n$ or $n+1 \leq i \leq 2n$ and   $j=\pi (i-n)$. 
The mapping that embeds $f(\pi)$ into $f(\rho)$ as an induced subgraph will be denoted by $a_i' \mapsto a_{e(i)}$, $b_i' \mapsto b_{w(i)}$. 

Firstly, observe that all but at most one entry from the set $\{w(1), w(2), \ldots, w(n)\}$ are less than or equal to $m$. 
Indeed, the vertices $b_{1}', b_{2}',  \ldots, b_{n}'$ have pairwise incomparable neighbourhoods, and this must also be the case for their images; 
however, if $i, j > m$, the neighbourhoods of $b_i$ and $b_j$ are comparable. Moreover, since $b_{i+1}'$ has two private neighbours with respect to $b_i'$ for any
$i \leq n-1$, we must have $w(i) < w(i+1)$ for any $i \leq n-1$, and hence we must have $w(1) < w(2) < \ldots < w(n-1) \leq m$ and $w(n-1) < w(n)$.    
Similarly, we can deduce that $m+1 \leq e(n+2)<e(n+3)<\ldots<e(2n)$ with $e(n+1)<e(n+2)$.  

Now, $a_1', a_2', \ldots, a_{n-2}'$ are adjacent to two vertices $b'_{n-2}, b'_{n-1}$ with $w(n-2) < w(n-1) \leq m$. 
Therefore, we conclude that $\{e(1), e(2), \ldots, e(n-2)\}$ must all be smaller than or equal to $m$.  As $a_1, a_2, \ldots, a_m$ form a chain graph together with the vertices in $B$, in order to have 
$N(a_{e(i)})\supsetneq N(a_{e(j)})$ for $1 \leq i <  j \leq n-2$, we conclude that we must have $1 \leq e(1) < e(2) < \ldots < e(n-2) \leq m$. 
To preserve correct adjacencies between $\{a_1', \ldots, a_{n-2}'\}$ and $\{b'_1, \ldots, b_{n-1}'\}$, we must have  
$$e(1) \leq w(1)<e(2) \leq w(2) < \ldots < e(n-2) \leq w(n-2) < w(n-1) \leq m.$$ 
Now $b_{w(n-1)}$ is already adjacent to $a_{e(1)}, a_{e(2)}, \ldots, a_{e(n-2)}$, but it has to be adjacent to two more vertices, $a_{e(n-1)}$ and $a_{e(n+\pi^{-1}(n-1))}$. 
Clearly, at least one of $e(n+\pi^{-1}(n-1))$ and $e(n-1)$ must be at most $w(n-1)$. Hence there are two cases: 
either both $e(n+\pi^{-1}(n-1))$ and $e(n-1)$ are at most $w(n-1)$, or one of them is at most  $w(n-1)$ and the other is at least $m+1$, in which case $e(n-1)$ is the one that is at most $w(n-1)$, as  $a'_{n-1}$ has a 
private neighbour with respect to $a'_{n+\pi^{-1}(n-1)}$. In either case, we must have $e(n-1) \leq w(n-1)$. 
As $a_{n-1}'$ is non-adjacent to $b_{n-2}'$, we must also have $w(n-2)<e(n-1)$, implying that
$$e(1) \leq w(1)<e(2) \leq w(2) < \ldots < e(n-2) \leq w(n-2) < e(n-1) \leq w(n-1) \leq m.$$ 

By symmetry, we derive that
$$m+1 \leq e(n+2) \leq w(n+2) < e(n+3) \leq \ldots <e(2n) \leq w(2n).$$ 

We are only left with determining the location of the embeddings of the four vertices $a'_n, b'_n, a'_{n+1}$, $b'_{n+1}$. 
Since $\pi(1) \neq n$, we have that $a'_{n+1}$ is not connected to $b'_n$, but connected to $b'_{\pi(1)}$ (with $\pi(1)<n$). 
It follows that $e(n+1) \geq m+1$. Clearly, for $a_{n+1}'$ to have two private neighbours with respect to $a_{n+2}'$ we must also have $e(n+1) < e(n+2)$.
The two private neighbours of $a'_{n+1}$ are $b_{\pi(1)}'$ and $b_{n+1}'$; since $a_{e(n+1)}$ only has one neighbour $b_i$ with $i < e(n + 1)$ (namely $b_{\pi(1)}$), the embedding of $b_{n+1}'$ must satisfy $e(n+1) \leq w(n+1)<e(n+2)$. 
Now $b_n'$, which is not adjacent to $a_{n+1}'$ but adjacent to $a'_{n+\pi^{-1}(n)}$ (note $e(n + \pi^{-1}(n)) \geq m + 1$ since $\pi^{-1}(n) > 1$) must therefore satisfy 
$w(n) \leq m$. As $b_n'$ has two private neighbours with respect to $b_{n-1}'$, we must have $w(n-1) < w(n)$, and as above, the private neighbour $a_n'$
of $b_n'$ must satisfy $w(n-1) < e(n) \leq w(n)$. 
Summarizing, we conclude that
$$e(1) \leq w(1) < \ldots < e(n) \leq w(n) \leq m < m+1 \leq e(n+1) \leq w(n+1) < \ldots < e(2n) \leq w(2n).$$ 

We may now alter this embedding of $f(\pi)$ into $f(\rho)$ if necessary to guarantee that $e(i)=w(i)$ for all $i=1,2, \ldots, 2n$. 
Indeed, it follows from the above inequalities that, for $1 \leq i \leq n$, $a_{e(i)}$ and $a_{w(i)}$ have the same set of neighbours 
among the embedded $b$-vertices, and similarly, for $n + 1 \leq i \leq 2n$, $b_{w(i)}$ and $b_{e(i)}$ have the same set of neighbours 
among the embedded $a$-vertices. We may thus keep the embeddings of $b'_1, \dots, b'_n, a'_{n + 1}, \dots, a'_{2n}$ where they are, 
and move the embeddings of the remaining vertices as appropriate to ensure $e(i) = w(i)$ for $1 \leq i \leq 2n$. 
From this altered embedding, it is easy to see that $\pi \subseteq \rho$ as claimed (for instance, interpret the matching between 
$b_1, \dots, b_m$ and $a_{m + 1}, \dots, a_{2m}$ as a line segment intersection model for $\rho$, and note that the intersection of 
this matching with the embedded graph $f(\pi)$ gives a line segment intersection model for $\pi$). 
\end{proof}

Claim~\ref{claim:3} cannot, in general, be extended to permutations $\pi$ with $\pi(1)=n$ (except trivially, when $n = 1$ or $m = n$). For example, if $\pi=(2,1)$ and $\rho=(1,2,3,4)$, then 
one can easily see that $f(\rho) \supseteq f(\pi)$, but $\rho$ does not contain $\pi$. 
One underlying reason for this phenomenon is that whenever $\pi(1)=n$, the vertices $a_n$ and $a_{n+1}$ have exactly the same neighbourhoods, 
which makes it possible for the graphs to be embedded with more flexibility, not necessarily forcing embedding of permutations. 
For this reason, we introduce a slight modification of the embedding, which allows us to always avoid the case $\pi(1)=n$. 

\begin{definition}
Given a permutation $\pi=(\pi(1), \pi(2), \ldots, \pi(n))$, define $\pi^*=(1, \pi(1)+1, \pi(2)+1, \ldots, \pi(n)+1)$. 
Define $f^*(\pi) = f(\pi^*)$, where $f$ is the map from permutations to quasi-permutation graphs. 
\end{definition}

\begin{theorem}\label{thm:perm}
The mapping $f^*$ is an injection from the class of permutations to the class of quasi-permutation graphs such that 
for any two permutations $\pi$ and $\rho$ we have $f^*(\pi) \subseteq f^*(\rho)$  if and only if  $\pi \subseteq \rho$. 
\end{theorem}
\begin{proof}
The mapping $f^*$ is a composition of two injective maps $\pi \mapsto \pi^*$ and $\pi^* \mapsto f(\pi^*)$, with the image of the second map being a quasi-permutation graph.
Therefore, $f^*$ is  an injection from the class of permutations to the class of quasi-permutation graphs. 
Further, $f^*(\pi) \subseteq f^*(\rho)$ means, by definition, that $f(\pi^*) \subseteq f(\rho^*)$, which happens if and only if $\pi^* \subseteq \rho^*$
(this follows from Claim~\ref{claim:3} as $\pi^*(1)=1 \neq n$). Finally, it is easy to see that $\pi^*\subseteq \rho^*$ if and only if $\pi \subseteq \rho$, from which the second part of the theorem follows. 
\end{proof}

%%%%%%%%%%%%%%%%%%%%%%%%%%%
\section{The structure of quasi-chain graphs} 
\label{sec:struct}
%%%%%%%%%%%%%%%%%%%%%%%%%%%

For two graphs $G_1=(V,E_1)$ and $G_2=(V,E_2)$ on the same vertex set we denote by $G_1\otimes G_2$ the graph $G=(V,E_1\otimes E_2)$,
where $\otimes$ denotes the symmetric difference of two sets. The main result in this section is  the following theorem.

\begin{theorem}\label{thm:decomposition}
If a bipartite graph $G=(A,B,E)$ is a quasi-chain graph, then $G=Z\otimes H$ for a chain graph $Z$ and a graph $H$ of vertex degree at most two
such that $E(H)\cap E(Z)$ and $E(H)-E(Z)$ are matchings. Such a decomposition  $G=Z\otimes H$ can be obtained in polynomial time.
\end{theorem}

In the proof of this result, we use a word representation for our graphs, which builds on a special case of {\em letter graph representations}, 
introduced in \cite{lettericity} (see Section~\ref{sec:let} for more details). The starting point is as follows: there is a bijective, 
order-preserving mapping between words over the alphabet $\{a, b\}$ (under the subword relation) and coloured chain graphs (under the coloured induced subgraph relation). 
This mapping sends a word $w$ to the graph whose vertices are the entries of $w$, and we have edges between each $a$ and each $b$ appearing after it in $w$. 
See Figure~\ref{fig:letex} for an example (the indices of the letters indicate the order of their appearance in $w$).

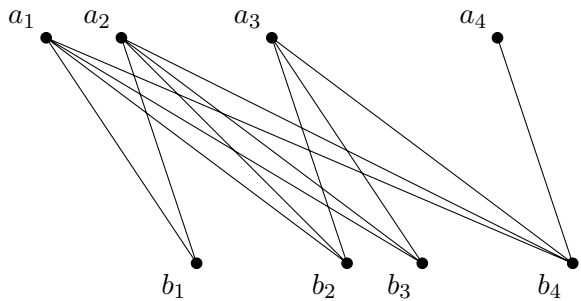
\begin{figure}[ht]
\centering
\begin{tikzpicture}
		\filldraw 
			(0, 3) circle(2pt) node[above left] {$a_1$}
			(1, 3) circle(2pt) node[above left] {$a_2$}
			(2, 0) circle(2pt) node[below left] {$b_1$}
			(3, 3) circle(2pt) node[above left] {$a_3$}
			(4, 0) circle(2pt) node[below left] {$b_2$}
			(5, 0) circle(2pt) node[below left] {$b_3$}
			(6, 3) circle(2pt) node[above left] {$a_4$}
			(7, 0) circle(2pt) node[below left] {$b_4$};
			
		\draw
			(0, 3) -- (2, 0)
			(0, 3) -- (4, 0)
			(0, 3) -- (5, 0)
			(0, 3) -- (7, 0)
			(1, 3) -- (2, 0)
			(1, 3) -- (4, 0)
			(1, 3) -- (5, 0)
			(1, 3) -- (7, 0)
			(3, 3) -- (4, 0)
			(3, 3) -- (5, 0)
			(3, 3) -- (7, 0)
			(6, 3) -- (7, 0);		 
			 
\end{tikzpicture}
\caption{The graph corresponding to the word $w = aababbab$}
\label{fig:letex}	
\end{figure}

We would like to extend this representation to graphs with the structure claimed in Theorem~\ref{thm:decomposition}. 
To do so, we enhance the letter representation described above by allowing bottom edges between pairs $a, b$ with 
the $a$ appearing before the $b$ in $w$ and top edges between pairs $a, b$ with the $a$ appearing after the $b$ in $w$. 
We require, in addition, that the set of top edges forms a matching and the set of bottom edges forms a matching, and interpret 
the bottom edges as an instruction to remove the corresponding matching from the chain graph represented by $w$, 
and the top edges as an instruction to add the corresponding matching. We call such a word an {\em enhanced word}.
For instance, $w'=\underline{aab}a\overline{bba}b$ is an enhanced word obtained from $w=aababbab$ by adding the bottom edge connecting the first $a$ to the first $b$
and the top edge connecting the second $b$ to the last $a$.

If $G$ is the graph described by an enhanced word $w$, we say $w$ is an {\em enhanced letter representation} for $G$.
In particular,  $w'=\underline{aab}a\overline{bba}b$ is an enhanced letter representation of the graph obtained from the graph in Figure~\ref{fig:letex} by removing the edge $a_1b_1$ 
and adding the edge $b_2a_4$. It is immediate from our discussion that Theorem~\ref{thm:decomposition} can be restated as follows.

\begin{theorem} \label{thm:enhanced-let}
Any quasi-chain graph admits an enhanced letter representation that can be found in polynomial time.
\end{theorem}

\begin{proof}
		At the core of our proof is an induction on the number of vertices of the quasi-chain graph $G$.
		The base case of the induction is trivial. To develop an inductive step, we prove the following claim.
		
		\begin{claim}\label{lem:deg1}
			Let $G = (A, B, E)$ be a quasi-chain graph. Then either $G$ or its bipartite complement has a vertex of degree at most 1. 
		\end{claim}
		
		\begin{proof}
			Let $a_1, \dots, a_t$ be the vertices of $A$ in a non-increasing order of their degrees. 
			If $a_1$ has fewer than 2 non-neighbours, we are done (since $a_1$ then has degree at most one in the bipartite complement). 
			Otherwise, let $b, b'$ be two non-neighbours of $a_1$. Note that $b$ and $b'$ have no common neighbour: 
			if $a$ was a common neighbour, then it would have two private neighbours with respect to $a_1$; 
			since $2P_3$s are forbidden, $a$ would be adjacent to all but at most one of the neighbours of $a_1$, from which $\deg(a) > \deg(a_1)$, contradicting our premise. 
			But then at least one of $b$ and $b'$ has degree at most one, since otherwise an induced $2P_3$ appears.
		\end{proof}

		Since the existence of enhanced letter representations is invariant under bipartite complementation and reflection (swapping the parts), 
		we may assume, by reflecting and complementing if necessary, that $G = (A, B, E)$ has a vertex $y$ of degree at most 1, and that $y \in B$. 
		
		Now our induction hypothesis says that $G' := G[A \cup (B - \{y\})]$ admits an enhanced letter representation $w'$. 
		If $y$ is isolated in $G$, we may always produce a representation $w$ for $G$ by adding $b$ as a prefix to $w'$. 
		The difficult case is when $y$ has degree 1 in $G$. Even then, we may easily produce a representation for $G$ 
		by adding $b$ as a prefix to $w'$ and linking it with a top edge to (the letter corresponding to) the vertex 
		$x$ that $y$ is pendant to, {\em provided that $x$ does not already have an incident top edge in $w'$}. 
		In the rest of the proof we show that $G'$ admits an enhanced letter representation in which $x$ is not incident to a top edge.

		To show this, we first observe that the mapping from enhanced letter representations to graphs is not injective.
		As a very simple example, the enhanced words $ab$ and $\overline{ba}$ both represent the complete graph on two vertices, 
		while $ba$ and $\underline{ab}$ both represent the edgeless graph on two vertices. 
		In general, we may swap the above pairs when the two letters appear next to each other. 
		We may also swap consecutive instances of the same letter, carrying over the top/bottom edges incident to them, 
		e.g., we may go from $\overline{ba}aa\underline{ab}$ to $\overline{b\underline{aaa}}\underline{\overline{\vphantom{b}aa}b}$ and vice-versa.

		To prove the result, we assume, by contradiction, that in any enhanced letter representation of $G'$ vertex $x$ 
		is incident to a top edge. Among all representations of $G'$, look at the ones that minimise the distance between $x$ and its top-matched neighbour. 
		Among those representations, pick one where the interval between $x$ and its top-matched neighbour has the minimum number of bottom edges. 
		Write $w^*$ for this representation, and denote by $y'$ the vertex top-matched to $x$. Given two letters $\alpha$ and $\beta$ in $w^*$ (two vertices in $G'$),
		we write $\alpha<\beta$ to indicate that $\alpha$ appears before  $\beta$ in the word, and denote by $\alpha-\beta$ the interval of letters (vertices) that appear {\it strictly between}
		$\alpha$ and $\beta$ in $w^*$. In particular, $y'<x$, since $y' \in B$, $x \in A$ and they are top-matched.
		We now derive a number of conclusions about the interval $y'-x$.
		
		\begin{itemize}
			\item[(1)] {\em The interval $y'-x$ is not empty}, since otherwise we could remove the top edge by 
			swapping $y'$ and $x$, and due to its minimality, {\em this interval starts with an $a$, which we denote $a^*$, and ends with a $b$, which we denote $b^*$}.
			
			\item[(2)] {\em The interval $y'-x$ does not contain $abb$ as an enhanced subword},	since otherwise the vertices corresponding to the $abb$ together with the vertices $x, y$ and $y'$ induce a $2P_3$ in $G$.
			
			\item[(3)] {\em The interval $y'-x$ contains at most two $b$s,} which follows directly from (1) and (2).
			% \item[(4)] {\em The interval $y'-x$ contains at most one bottom edge}. Indeed, if it contains two bottom edges, say $\underline{a'b'}$ and $\underline{a''b''}$
			% with $b'<b''$, then by (2) the interval between $a'-b'$ does not contain any $b$, in which case 
			% the bottom edge  $\underline{a'b'}$ can be removed by bringing $a'$ next to $b'$ and swapping them.   
		\end{itemize}
		\noindent
		To obtain a contradiction, we analyze the following two cases. 
		
		\medskip
		{\it Case 1: $a^*$ and $b^*$ are not bottom-matched.} Then there is no $b$ in the interval $a^*-b^*$. Indeed, if $b'$ belongs to this interval, 
		then, according to (2), $a^*$ is bottom-matched to $b'$. However, this contradicts the choice of $w^*$, because,
		according to (3), this bottom edge can be removed by bringing $a^*$ next to $b'$ and swapping them. 
		In a similar way, in the absence of a second $b$, any bottom edge can be removed from the interval $y'-x$, implying that this interval has no bottom edges.     
		
		We note that at least one of $b^*$ and $x$ must have a bottom-matched neighbour, 
		since otherwise we could reduce the interval by swapping $b^*$ and $x$ and introducing the bottom edge between them.
		If $x$ has a bottom-matched neighbour, then $x, y, y'$ together with $a^*, b^*$ and the bottom-matched neighbour of $x$ induce a $2P_3$.
		Therefore, $b^*$ has a bottom-matched neighbour $a'$ with $a'<y'$. 
		
		We also note that at least one of $a^*$ and $b^*$ must have a top-matched neighbour, 
		since otherwise we could bring $a^*$ next to $b^*$, swap them by introducing a top edge, and then reduce the interval by swapping $a^*$ and $x$.
		If $b^*$ has a top-matched neighbour, then $y', a', x$ together with $b^*, a^*$ and the top-matched neighbour of $b^*$ induce a $2P_3$. 
		If $a^*$ is has a top-matched neighbour, then $x, y, y'$ together with $a^*, b^*$ and the top-matched neighbour of $a^*$ induce another $2P_3$.
		
		\medskip
		{\it Case 2: $a^*$ and $b^*$ are bottom-matched.} Clearly, the interval $a^*-b^*$ is not empty, since otherwise we could remove the bottom edge by swapping  $a^*$ and $b^*$.
		Also, to avoid an easy reduction to Case 1, we conclude that the letter to the right of $a^*$ is a $b$ (we denote it by $b^\circ$), and 
		the letter to the left of $b^*$ is an $a$ (we denote it by $a^\circ$). 
		
		We note that either $a^*$ or $b^\circ$ is incident to a top edge, since otherwise we could swap them by introducing the top edge $\overline{b^\circ a^*}$
		and then reduce the interval $y'-x$ by swapping $y'$ and $b^\circ$. Similarly, at least one of $a^\circ$ and $b^*$ is incident to a top edge.
		
		If $a^*$ is incident to a top edge, then $x,y,y'$ together with $a^*,b^\circ$ and a top-matched neighbour of $a^*$ induce a $2P_3$.
		If $a^\circ$ is incident to a top edge, then $x,y,y'$ together with $a^\circ,b^*$ and a top-matched neighbour of $a^\circ$ induce a $2P_3$.
		Therefore, $b^\circ$ is top-matched with a vertex $a'$ and $b^*$ is incident to a top edge. 
		We can assume that $x<a'$, since otherwise we could remove the top edge between $b^\circ$ and $a'$ by bringing them next to each other and swapping.  
		But then $a^*,b^\circ,a'$ together with $a^\circ,b^*$ and a top-matched neighbour of $b^*$ induce a $2P_3$. 
		
		\medskip
		A contradiction in all cases shows that $G'$ admits an enhanced letter representation in which $x$ is not incident to a top edge and 
		completes the inductive step.

		\medskip
		% It remains to describe how an enhanced word representation for a quasi-chain graph $G$ can be found in polynomial time. 
		% One way to achieve this in at most quadratic time\footnote{While the analogous problem for chain graphs suggests there are likely more efficient ways, we leave the optimisation problem for future research.} 
		% is by inductively obtaining a representation for the graph $G'$, then if necessary, modifying it in linear time so that $x$ is not adjacent to a green edge. 
		% If we can indeed perform the modification in linear time, we are done: we may simply add $y$ to the start of the representation and connect it to $x$ with a green edge; 
		% any required complementation and reflection can be done in linear time. 
		
		Our case analysis leads to a polynomial-time procedure for removing, if necessary, the top edge incident to $x$, which can be outlined as follows. 
		The contradictions involving the appearance of a $2P_3$ concern cases that do not actually occur when we apply our procedure, so we ignore them. 
		When a contradiction to the minimality in the construction of $w^*$ appears in the case analysis, 
		we repeatedly execute the operation that lead to the contradiction -- we only need to iterate a linear number of times. 
		We invariably arrive at the situation where $y'$ and $x$ appear next to each other, and we simply swap them to remove the top edge.
\end{proof}

To conclude the section, we observe that the converse to Theorem~\ref{thm:enhanced-let} does not hold. In particular, $2P_3$ has 8 different 
enhanced letter graph representations (4 per colouring), up to moving the top/bottom edges between twin vertices.

%%%%%%%%%%%%%%%%%%%%%%%%%%%%%%%%%%%%%%%%%%%%%%%%%%%%%%%%%%%%%%%%%%%%%%%%%%%%%%%%%%%%%%%%%%%%%%%%%%%%%%%%%%%%%
%%%%%%%%%%%%%%%%%%%%%%%%%%%%%%%%%%%%%%%%%%%%%%%%%%%%%%%%%%%%%%%%%%%%%%%%%%%%%%%%%%%%%%%%%%%%%%%%%%%%%%%%%%%%%

\section{Well-quasi-orderability and lettericity in the class of quasi-chain graphs} 
\label{sec:let}

Let $(X, \leq)$ be a poset. As a quick refresher, a {\em chain} is a set of pairwise comparable elements, 
and an {\em antichain} is a set of pairwise incomparable elements. 
$X$ is said to be {\em well-quasi-ordered} by $\leq$ (``wqo'' for short) if there are no infinite strictly descending chains, 
and no infinite antichains in $(X, \leq)$.\footnote{We note that the condition on strictly descending chains 
is trivially satisfied for finite graphs, so it suffices to investigate the presence of infinite antichains.} 
Well-quasi-orderability in the universe of graphs has received much attention, culminating in the celebrated 
result of Robertson and Seymour that graphs are wqo by the minor relation \cite{robertson-seymour}. 
When considering the induced subgraph relation instead, finding infinite antichains is easy (the cycles are an example). 
However, the story is far from over: a challenging problem is to characterise those hereditary classes that are wqo. 
The last few decades have witnessed a slow but steady effort in this direction (see, for instance, \cite{cographs, wqo, bigenic, lettericity}). 

It is shown in \cite{wqo} that quasi-chain graphs are not wqo under the induced subgraph relation 
(and indeed, this also follows directly from Theorem~\ref{thm:perm}, since permutations are not wqo -- see, e.g., 
\cite{antichain}). We start this section by providing a simple, explicit example of an infinite antichain in this class, 
which is independent of the relationship between quasi-chain graphs and permutations.
 
Let $Z_n$ be the universal chain graph on $2n$ vertices, with the labelling given in Figure~\ref{fig-chain}. 
Now let $Q_n$ be the graph obtained from $Z_n$ by deleting all edges of the form $(a_i, b_{i + 1})$ 
(those edges form a matching), then adding a pendant vertex to each of $a_1$ and $b_n$, as shown in Figure~\ref{fig-qgraph}.

\begin{figure}[ht]
	\centering
	\begin{subfigure}[t]{1\linewidth}
		\centering
		\begin{tikzpicture}[scale=1, transform shape]
			\foreach \i in {1,...,6} {
				\filldraw (\i * 2, 0) circle (2pt) node[below right]{$b_{\i}$};
				\filldraw (\i * 2, 2) circle (2pt) node[above left]{$a_{\i}$};
				\foreach \x in {\i,...,6} {
					\draw (\i * 2, 2) -- (\x * 2, 0);
				}
			}
		\end{tikzpicture}
		\captionsetup{justification=centering}
		\caption{The universal chain graph $Z_6$}
		\label{fig-chain}	
	\end{subfigure}	
	
	\bigskip
	
	\bigskip
	
	\begin{subfigure}[t]{1\linewidth}
		\centering
		\begin{tikzpicture}[scale=1, transform shape]
			\foreach \i in {1,...,6} {
				\draw (\i * 2, 2) -- (\i * 2, 0);
				\filldraw (\i * 2, 0) circle (2pt) node[below right]{$b_{\i}$};
				\filldraw (\i * 2, 2) circle (2pt) node[above left]{$a_{\i}$};
			}
			
			\foreach \i in {1,...,4} {
				\foreach \x in {\i,...,4} {
					\draw (\i * 2, 2) -- (4 + \x * 2, 0);
				}
			}
			\filldraw (1, 0) circle (2pt) node[below left]{$b'_{1}$};
			\filldraw (13, 2) circle (2pt) node[above right]{$a'_{6}$};
			%\filldraw (0, 0) circle (2pt) node[below left]{$b''_{1}$};
			%\filldraw (14, 2) circle (2pt) node[above right]{$a''_{6}$};
			\draw (13, 2) -- (12, 0);
			\draw (2, 2) -- (1, 0);
			%\draw (14, 2) -- (12, 0);
			%\draw (2, 2) -- (0, 0);
		\end{tikzpicture}
		\captionsetup{justification=centering}
		\caption{The graph $Q_6$ obtained from it}
		\label{fig-qgraph}	
	\end{subfigure}
	\caption{An infinite antichain of quasi-chain graphs}
\end{figure}
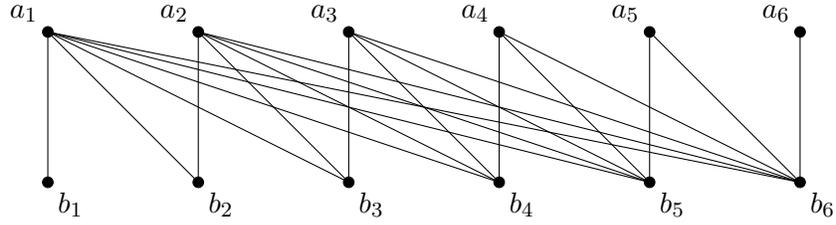
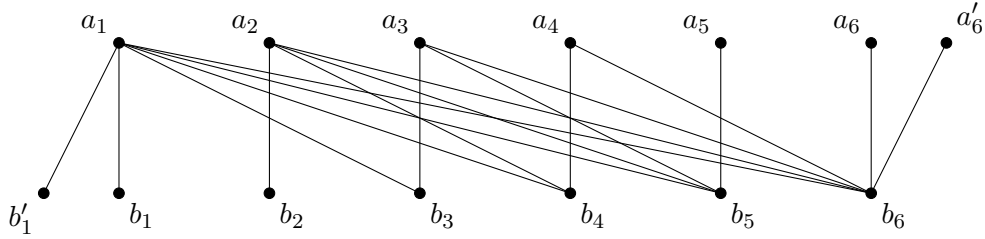

\begin{lemma}
	$(Q_k)_{k \geq 4}$ is an infinite antichain of quasi-chain graphs with respect to the induced subgraph relation.
\end{lemma}
\begin{proof}
	First, note that the graphs are indeed quasi-chain. This follows from the fact that the ordering $a_1, a_2, \dots, a_n, a'_n$ is good 
(and, by symmetry, so is $b_n, b_{n - 1}, \dots, b_1, b'_1$). Indeed, for $i < j$, $a_j$ has at most one private neighbour with respect to $a_i$, namely $b_j$.
	
	To see that the sequence $(Q_k)_{k \geq 4}$ is an antichain, let $4 \leq m \leq n$, 
and label the vertices of $Q_m$ as in Figure~\ref{fig-qgraph}, and the vertices of $Q_n$ by replacing $a$s with $\alpha$s and $b$s with $\beta$s. 
Suppose $\iota : Q_m \to Q_n$ is an induced subgraph embedding. By symmetry and connectedness of $Q_m$, we may assume $\iota$ maps $a$-vertices to $\alpha$-vertices and $b$-vertices to $\beta$-vertices, respectively. 

Among ordered pairs of $\alpha$-vertices with incomparable neighbourhoods, $(\alpha_1, \alpha_2)$ is the {\em only} one where the first vertex has 3 private neighbours with respect to the second. 
This fact immediately forces $\iota(a_1) = \alpha_1$ and $\iota(a_2) = \alpha_2$. But then  
\begin{itemize}
\item[] $\iota(b_2) = \beta_2$, since $\beta_2$ is the only $\beta$-vertex non-adjacent to $\alpha_1$, implying that
\item[] $\iota(b_3) = \beta_3$, since otherwise the image of $b_3$ has no candidate neighbour for the image of $a_3$, implying that $b_1,b'_1$ are mapped to $\beta_1,\beta'_1$, implying that
\item[] $\iota(a_3) = \alpha_3$, since $\alpha_3$ is the only neighbour of $\beta_3$ among not yet mapped vertices, implying that 
\item[] $\iota(b_4) = \beta_4$, since $\beta_4$ is the only $\beta$-vertex non-adjacent to $\alpha_3$ among not yet mapped vertices, etc.
\end{itemize}
Proceeding in this way, we conclude that $\iota(a_i) = \alpha_i$ and $\iota(b_i) = \beta_i$ for all $i\le m$, which is possible only if $m=n$.
\end{proof}

Knowing that the full class of quasi-chain graphs is not wqo, a natural question is to determine exactly what the obstacles to wqo are in this class. 
This is a challenging problem and as a first step towards its solution we analyze  the {\em lettericity} of quasi-chain graphs. In the context of wqo, 
the importance of this parameter is due to the fact that bounded lettericity implies wqo by induced subgraphs \cite{lettericity}. The parameter is defined as follows.

Let $\Omega$ be a finite alphabet and ${\cal P}\subseteq \Omega^2$ a set of ordered pairs of symbols from $\Omega$, called the {\it decoder}.
To each word $w=w_1w_2\cdots w_n$ with $w_i\in \Omega$ we associate a graph $G({\cal P},w)$, called the {\it letter graph}
of $w$, by defining $V(G({\cal P},w))=\{1,2,\ldots ,n\}$ with $i$ being adjacent to $j>i$ if and only if 
the ordered pair $(w_i,w_j)$ belongs to the decoder $\cal P$. 

It is not difficult to see that every graph $G$ is a letter graph in an alphabet of size at most $|V(G)|$ over an appropriate decoder $\cal P$.
The minimum $k$ such that $G$ is a letter graph in an alphabet of $k$ letters is the {\it lettericity} of $G$ and is denoted
$\lett(G)$. A graph is a $k$-letter graph if its lettericity is at most $k$.  

\medskip

In what follows, the class of graphs of vertex degree at most 1 (that is, induced matchings) plays an important role, and so does the class of their bipartite complements. 
We denote those classes by $\mathcal{M}$ and $\widetilde{\mathcal M}$ respectively. 

We will need a few basic facts about lettericity that we summarise here without proof (all of those facts are shown in \cite{lettericity}, except the minimality in Fact~\ref{fact4} -- can be easily shown directly).

\begin{fact} \label{fact1}
	Any class of graphs of bounded lettericity is wqo.
\end{fact}
\begin{fact} \label{fact2}
	For any graph $G$ and vertex $x$ of $G$, $\lett(G) \leq 2\lett(G - x) + 1$.
\end{fact}
\begin{fact} \label{fact3}
	Chain graphs have lettericity at most 2 (see Section~\ref{sec:struct}).
\end{fact}
\begin{fact} \label{fact4}
	The classes $\mathcal M$ and $\widetilde{\mathcal M}$ are minimal hereditary classes of unbounded lettericity.
\end{fact}

\medskip

We claim that, in addition to the classes $\mathcal M$ and $\widetilde{\mathcal M}$, there is only one more minimal class of unbounded lettericity among quasi-chain graphs, 
defined as follows. As before, let $Z_n$ be the prime chain graph on $2n$ vertices illustrated in Figure~\ref{fig-chain}. 
We construct {\em double-chain} graphs $D_n$ as follows: start with $Z_{3n}$, then like in the construction of $Q_{3n}$, 
delete all edges of the form $(a_i, b_{i + 1})$. Finally, delete all vertices whose index is divisible by 3. 
$D_n$ can be thought of as $Z_n$, where we replace each vertical edge with a $2P_2$ -- see Figure~\ref{fig-doublechain} for an illustration.

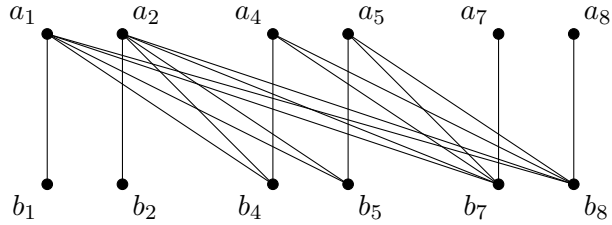
\begin{figure}
	\centering
	\begin{tikzpicture}[scale=1, transform shape]
		
		\foreach \i in {1, 4, 7} {
			\filldraw (\i, 2) circle (2pt) node[above left]{$a_{\i}$};
			\filldraw (\i, 0) circle (2pt) node[below left]{$b_{\i}$};	
		}
		
		\foreach \i in {2, 5, 8} {
			\filldraw (\i, 2) circle (2pt) node[above right]{$a_{\i}$};
			\filldraw (\i, 0) circle (2pt) node[below right]{$b_{\i}$};	
		}
		
		\foreach \i in {1, 2, 4, 5, 7, 8} {
			\draw (\i, 2) -- (\i, 0);	
		}
		
		\foreach \i in {1, 2} {
			\foreach \x in {4, 5, 7, 8} {
				\draw (\i, 2) -- (\x, 0);	
			}	
		}
		
		\foreach \i in {4, 5} {
			\foreach \x in {7, 8} {
				\draw (\i, 2) -- (\x, 0);	
			}	
		}

	\end{tikzpicture}
	\caption{The double-chain graph $D_3$}
	\label{fig-doublechain} 
\end{figure}

Let $\mathcal D$ be the class containing, for each value of $n$, the graph $D_n$ and all of their induced subgraphs. 
We note that the chain ordering inherited from the starting graph $Z_{3n}$ is good in $D_n$, so that $\mathcal D$ is indeed a subclass of quasi-chain graphs.

\begin{lemma}
	$\mathcal D$ is a minimal hereditary class of unbounded lettericity.
\end{lemma}

\begin{proof}
	We first show that any proper subclass of $\mathcal D$ has bounded lettericity. 
Indeed, such a subclass is $D_n$-free for an appropriately large $n$, and any $D_n$-free graph $G$ contains at most $n$  copies of induced $2P_2$s. 
This means we may remove at most $4n$ vertices from $G$ to obtain a chain graph. Fact~\ref{fact3} and repeated application of Fact~\ref{fact2} gives a bound on $\lett(G)$ that only depends on $n$.
	
	It remains to show that lettericity is unbounded in $\mathcal D$. To see this, suppose for a contradiction that the lettericity is bounded by $k$. 
The graph $D_n$ consists of $n$ copies of induced $2P_2$s connected in a chainlike manner. 
Given a $k$-letter word $w$ representing $D_n$, we consider the subwords of $w$ representing each of the $2P_2$s. 
In particular, by the pigeonhole principle, for any $t \in \mathbb N$, we may find an $N$ large enough such that $t$ of the $2P_2$s in $D_N$ are represented by the same subword. 
Those $t$ copies of $2P_2$s induce a copy of $D_t$ in $D_N$ whose letter graph representation only uses 4 letters; 
in particular, since any $D_t$ has such a representation, we may assume $k \leq 4$. 
A similar argument shows that for each $D_n$ there must exist a representation with letters $a, b, c, d$, 
where the four respective letter classes are (using the indexing from Figure~\ref{fig-doublechain}) 
$A := \{a_i : i = 1 \!\! \mod 3\}$, $B := \{b_i : i = 1 \!\! \mod 3\}$, $C := \{a_i : i = 2 \!\! \mod 3\}$ and $D := \{b_i : i = 2 \!\! \mod 3\}$.
 Standard arguments show that, up to symmetry, the decoder for this representation must be $\{(a, b), (a, d), (c, b), (c, d)\}$. 
But even a single $2P_2$ cannot be expressed in this way -- a contradiction.
\end{proof}

We are ready for the main result of this section, which characterises classes of bounded lettericity among quasi-chain graphs.
In the proof, given two vertex-disjoint bipartite graphs $G_1 = (A_1, B_1, E_1)$ and $G_2 = (A_2, B_2, E_2)$, we define the {\em skew-join} of $G_1$ with $G_2$ 
as the graph $(A_1 \cup A_2, B_1 \cup B_2, E_1 \cup E_2 \cup A_1 \times B_2)$.

\begin{theorem}\label{thm:boundedlet}
	Let $\mathcal X$ be a hereditary subclass of quasi-chain graphs. Then $\mathcal X$ has bounded lettericity if and only if 
$\mathcal X$ excludes at least one graph from each of $\mathcal M, \widetilde{\mathcal M}$ and $\mathcal D$.
\end{theorem}
\begin{proof}
	The ``only if'' direction is clear, since $\mathcal M, \widetilde{\mathcal M}$ and $\mathcal D$ all have unbounded lettericity. 
For the ``if'' direction, let $\mathcal X$ be a hereditary subclass of quasi-chain graphs excluding a graph from each of the three classes. 
It suffices to show that the classes $\mathcal X_{s, t, n}$ of $(sP_2, \widetilde{tP_2}, D_n)$-free quasi-chain graphs have bounded lettericity for all $s, t, n \in \mathbb N$, since $\mathcal X$ is contained in such a class.
	
We prove the statement by induction on $n$. The statement is clearly true if $n = 1$ for all $s, t$, since $\mathcal X_{s, t, 1}$ is a subclass of chain graphs, which have lettericity 2.
	
\medskip
Now suppose $n \geq 1$, and let $G = (A, B, E) \in  X_{s, t, n + 1}$. By Theorem~\ref{thm:decomposition}, $G = Z \otimes H$, where $Z$ is a chain graph, and $E(H) \cap E(Z)$, $E(H) - E(Z)$ are both matchings. 
	
Let $a_1, \dots, a_k$ be the vertices of $A$ listed in non-increasing order with respect to their neighbourhoods in $Z$. 
Each vertex $a_i$ gives a partition of $A$ into a ``left'' part $A^l_i = \{a_1, \dots, a_i\}$ and 
a ``right'' part $A^r_i = \{a_{i + 1}, \dots, a_k\}$, and a partition of $B$ into $B^l_i = B - N(a_i)$ and $B^r_i = N(a_i)$. 
This produces a cut of $Z$ into two smaller chain graphs $Z^l_i := Z[A^l_i \cup B^l_i]$ and $Z^r_i := Z[A^r_i \cup B^r_i]$, 
and it is not difficult to see $Z$ is the skew-join of $Z^l_i$ with $Z^r_i$, since $A^l_i$ is complete to $B^r_i$, 
while $A^r_i$ is anticomplete to $B^l_i$. Similarly, we obtain a cut of $G$ into quasi-chain graphs $G^l_i$ and $G^r_i$. 
We will refer to those cuts as the cuts {\em induced by $a_i$}. 
	
\medskip
These cuts are very neat in the chain graph $Z$, but how do they look in the original quasi-chain graph $G$? 
Specifically, where do induced $2P_2$s in $G$ appear with respect to these cuts? 
The first thing to note is that, for any given cut, the edges between $A^r_i$ and $B^l_i$ in $G$ belong to $E(H) - E(Z)$, 
and thus induce a matching. Since $G$ is $sP_2$-free, there are at most $s - 1$ of them. 
Similarly, there are at most $t - 1$ non-edges in $G$ between $A^l_i$ and $B^r_i$. 
We call the (at most $2s + 2t - 4$) vertices incident to those edges or non-edges {\em $i$-dirty}. 
We call an induced $2P_2$ in $G$ {\em $i$-bad} if it does not contain any $i$-dirty vertex 
(the reasoning being that the bad $2P_2$s do not simply disappear when removing dirty vertices). 
We now claim that any $i$-bad $2P_2$ lies completely in $G^l_i$ or in $G^r_i$ (we call it {\em left $i$-bad} 
or {\em right $i$-bad} accordingly). To see that this is indeed the case, we simply note that any $2P_2$ 
with vertices in both $G^l_i$ and $G^r_i$ needs to have either a crossing edge between $A^r$ and $B^l$, 
or a crossing non-edge between $A^l$ and $B^r$. Finally, we call the cut induced by $a_i$ {\em perfect} 
if there are no $i$-bad $2P_2$s, {\em good} if there is both a left $i$-bad $2P_2$ and a right $i$-bad $2P_2$, 
and {\em bad} if it neither good nor perfect. There are three possible cases:
	
\begin{itemize}
	\item[i)] {\em There is an $i$ such that the cut induced by $a_i$ is perfect.} 
In this case, we note that ``cleaning the cut'' by removing all $i$-dirty vertices 
from $G$ yields a chain graph $G'$. But we have removed a bounded number of vertices, 
hence Fact~\ref{fact3} and repeated application of Fact~\ref{fact2} give an upper bound on the lettericity of $G$ that only depends on $s$ and $t$.  
		
\item[ii)] {\em There is an $i$ such that the cut induced by $a_i$ is good.} 
Then like before, cleaning the cut yields a quasi-chain graph $G'$ which is 
a skew-join of the graphs $G'^l := G' \cap G^l$ and $G'^r := G' \cap G^r$. 
By construction, $G'^l$ and $G'^r$ each have a $2P_2$; since $G$ (and hence $G'$) is $D_{n + 1}$-free, 
it follows that $G'^l$ and $G'^r$ are both $D_n$-free, and the inductive hypothesis applies. 
From the representations of $G'^l$ and $G'^r$ with a bounded number of letters, 
it is easy to construct one for their skew-join $G'$, then use that representation to construct one for $G$ like in the previous case.
		
\item[iii)] {\em Every cut is bad.} This means that each $a_i$ has either a left or a right $i$-bad $2P_2$ (but not both). 
We note that $a_1$ must have a right $1$-bad $2P_2$, while $a_k$ must have a left $k$-bad $2P_2$. 
Moreover, if a $2P_2$ is left, respectively right $i$-bad, then it is left $j$-bad for any $j \geq i$, 
respectively right $j$-bad for any $j \leq i$. This implies that there is one specific $i_0$ such that 
$a_1, \dots, a_{i_0}$ all have right bad $2P_2$s, while $a_{i_0 + 1}, \dots, a_k$ all have left bad $2P_2$s. 
We claim that no $2P_2$ can be simultaneously $i_0$- and $i_0 + 1$-bad. Indeed, both vertices 
$a_{i_1}, a_{i_2} \in A$ of such a $2P_2$ would simultaneously need $i_1, i_2 > i_0$ and $i_1, i_2 \leq i_0 + 1$, which is impossible. 
It follows that cleaning both of the cuts induced by $a_{i_0}$ and $a_{i_0 + 1}$ leaves us with a chain graph, and we proceed as in the first case.
\end{itemize}
	
\end{proof}

Theorem~\ref{thm:boundedlet} gives us a characterisation of subclasses of quasi-chain graphs of bounded lettericity. 
All of those subclasses are wqo, but a wqo class need not have bounded lettericity -- for instance, 
the minimal classes $\mathcal M, \widetilde{\mathcal M}$ and $\mathcal D$ themselves are wqo. 
For $\mathcal M$ and $\widetilde{\mathcal M}$, this is a special case of Theorem~2 from \cite{bigenic}. Let us now show the claim for $\mathcal D$.

\begin{theorem} \label{lem:doublechainwqo}
	$\mathcal D$ is wqo by induced subgraphs.
\end{theorem}

\begin{proof}
	It suffices to produce an order-preserving surjection from a wqo poset $(X, \leq)$ to $\mathcal D$ ordered by the induced subgraph relation (this fact is standard -- see, e.g., \cite{monotone}, Proposition~3.1).
	
Our poset $X$ will be the set of words over a finite alphabet of incomparable letters, ordered under the subword relation -- wqo of this poset is a special case of Higman's Lemma. 
Note that a coloured $2P_2$ has, up to isomorphism, 9 distinct non-empty induced subgraphs. 
Consider an alphabet $\Omega$ consisting of incomparable letters $A_1, \dots, A_{9}$, 
where each letter corresponds (arbitrarily) to one of those induced subgraphs.
We define a map $\varphi$ from the set $\Omega^*$ of words over $\Omega$ to graphs inductively, 
by defining $\varphi(A_i)$ to be the corresponding induced subgraph of $2P_2$, 
and $\varphi(A_iw')$ to be the skew-join of $\varphi(A_i)$ with $\varphi(w')$ 
(where $A_iw'$ denotes the concatenation of $A_i$ with the word $w'$). 
	
We note that the image of any word of length $n$ is an induced subgraph of $D_n$ (see Figure~\ref{fig-doublechain}), 
hence $\varphi(\Omega^*) \subseteq \mathcal D$. Since any induced subgraph of $D_n$ can be obtained in this way, $\varphi$ is surjective. 
Finally, it is straightforward to check that $\varphi$ is order-preserving.  
\end{proof}

% We leave the characterisation of wqo subclasses of quasi-chain graphs as an open problem:

% \begin{problem}
% 	Characterise wqo subclasses of quasi-chain graphs.
% \end{problem}

%%%%%%%%%%%%%%%%%%%%%%%%%%%%%%%%%%%%%%%%%%%%%%%%%%%%%%%%%%%%%%%%%%%%%%%%%%%%%%%%%%%%%%%%%%%%%%%%%%%%%%%%%%%%%
%%%%%%%%%%%%%%%%%%%%%%%%%%%%%%%%%%%%%%%%%%%%%%%%%%%%%%%%%%%%%%%%%%%%%%%%%%%%%%%%%%%%%%%%%%%%%%%%%%%%%%%%%%%%%

\section{Implicit representation of quasi-chain graphs}
\label{sec:implicit}
%%%%%%%%%%%%%%%%%%%%%%%%%%%%%%%%%%%%%%%%%%%%%%%%%%%%%%%%%%%%%%%%%%%%%%%%%%%%%%%%%%%%%%%%%%%%%%%%%%%%%%%%%%%%%
%%%%%%%%%%%%%%%%%%%%%%%%%%%%%%%%%%%%%%%%%%%%%%%%%%%%%%%%%%%%%%%%%%%%%%%%%%%%%%%%%%%%%%%%%%%%%%%%%%%%%%%%%%%%%

The idea of implicit representation of graphs was introduced in \cite{implicit} and
	can be described as follows. A representation of an $n$-vertex graph $G$ is said to be implicit if
	it assigns to each vertex of $G$ a binary code of length $O(\log n)$ so that the adjacency of two
	vertices is a function of their codes. 
	
	Not every class of graphs admits an implicit representation, since a bound on the length of a vertex code implies a bound
	on the number of graphs admitting such a representation. More precisely, only classes containing $2^{O(n \log n)}$ labelled graphs with $n$
	vertices can admit an implicit representation. In the terminology of \cite{speed}, hereditary classes containing $2^{O(n \log n)}$ labelled graphs on $n$ vertices are at most factorial, i.e., they have
	at most factorial speed of growth. Whether all hereditary classes with at most factorial speed admit an implicit representation is a big open question known as the {\it implicit representation conjecture}.
	The conjecture holds for a variety of factorial classes such as interval graphs, permutation graphs (which include chain graphs), line graphs, planar graphs, etc. It also holds for 
	all graph classes of bounded vertex degree, of bounded clique-width, of bounded arboricity (including all proper minor-closed classes), etc.; see \cite{implicit-factorial} for more information on this topic.

	The class of $2P_3$-free bipartite graphs is known to be factorial, which was shown in \cite{Allen}. However, the question whether 
	this class admits an implicit representation remains open. In this section, we answer this question in the affirmative. 
	To this end, we introduce the following general tool.
	
	For a graph $G=(V,E)$, let $A_G$ denote the adjacency matrix of $G$, and for two vertices $x,y\in V$, let $A_G(x,y)$
	be the element of the matrix corresponding to $x$ and $y$. Given a Boolean function $f$ of $k$ variables and graphs $H_1=(V,E_1),\ldots,H_k=(V,E_k)$, 
	we will write $G=f(H_1,\ldots,H_k)$ if $$A_G(x,y)=f(A_{H_1}(x,y),\ldots,A_{H_k}(x,y))$$ for all distinct vertices $x,y\in V$. 
	If $G=f(H_1,\ldots,H_k)$, we say that $G$ is an $f$-function of $H_1,\ldots,H_k$.
	
	\begin{theorem}\label{thm:function}
		Let $X$ be a class of graphs, $k$ a natural number,  $f$ a Boolean function of $k$ variables, and $Y_1,\ldots,Y_k$ classes of 
		graphs admitting an implicit representation. If every graph in $X$ is an $f$-function of graphs $H_1\in Y_1,\ldots,H_k\in Y_k$,
		then $X$ also admits an implicit representation. 
	\end{theorem} 
	
	\begin{proof}
		To represent a graph $G=f(H_1,\ldots,H_k)$ in $X$ implicitly, we assign to each vertex of $G$ $k$ labels, each of which represents
		this vertex in one of the graphs $H_1,\ldots,H_k$. Given the labels of two vertices $x,y\in V(G)$, we can compute the adjacency of 
		these vertices in each of the $k$ graphs and hence, using the function $f$ (which we may encode in each label with a constant number of bits), we can compute the adjacency of $x$ and $y$ in the graph $G$. 
	\end{proof}

	According to Theorem~\ref{thm:decomposition}, any quasi-chain graph is a $\oplus$-function of a chain graph and a graph of vertex degree at most 2, where $\oplus$ is addition modulo 2. 
	As we mentioned earlier, chain graphs and graphs of vertex degree at most 2 admit an implicit representation. Together with Theorem~\ref{thm:function}
	this implies the following conclusion. 
	\begin{corollary}\label{cor:implicit}
		The class of quasi-chain graphs admits an implicit representation. 
	\end{corollary}

	The same conclusion can be derived in an alternative way, which is of independent interest, because it deals with a parameter motivated by some biological applications.
	This parameter was introduced in \cite{contiguity} under the name contiguity and it can be defined as follows. 
	
	Graphs of contiguity $1$ are graphs that admit a linear order of the vertices in which the neighbourhood of each vertex forms an interval. 
	Not every graph admits such an ordering, in which case one can relax this requirement by looking for an ordering in which the neighbourhood 
	of each vertex can be split into at mots $k$ intervals. The minimum value of $k$ which allows a graph $G$ to be represented in this way is the {\it contiguity} of $G$.

	\begin{theorem}\label{thm:contiguity}
		Contiguity of quasi-chain graphs is at most 3.
	\end{theorem}
	
	\begin{proof}
		It is not difficult to see that chain graphs have contiguity $1$. Let $G$ be a quasi-chain graph, and use Theorem~\ref{thm:decomposition} to obtain a decomposition $G = Z \otimes H$. Consider a linear order of the vertices of $G$ such that their neighbourhoods in $Z$ are intervals. $Z$ can be transformed into $G$ by adding
		at most one edge and at most one non-edge incident to each vertex. By adding a non-edge, we split the interval of neighbours of $v$ into at most two intervals, and by adding a neighbour to $v$, its neighbourhood spans at most one additional interval
		consisting of a single vertex.
	\end{proof}
	
	It is not difficult to see that graphs of bounded contiguity admit an implicit representation. Therefore, Corollary~\ref{cor:implicit} follows from Theorem~\ref{thm:contiguity} as well.
%%%%%%%%%%%%%%%%%%%%%%%%%%%%%%%%%%%%%%%%%%%%%%%%%%%%%%%%%%%%%%%%%%%%%%%%%%%%%%%%%%%%%%%%%%%%%%%%%%%%%%%%%%%%%
%%%%%%%%%%%%%%%%%%%%%%%%%%%%%%%%%%%%%%%%%%%%%%%%%%%%%%%%%%%%%%%%%%%%%%%%%%%%%%%%%%%%%%%%%%%%%%%%%%%%%%%%%%%%%

\section{Optimisation in quasi-chain graphs}
\label{sec:opt}
%%%%%%%%%%%%%%%%%%%%%%%%%%%%%%%%%%%%%%%%%%%%%%%%%%%%%%%%%%%%%%%%%%%%%%%%%%%%%%%%%%%%%%%%%%%%%%%%%%%%%%%%%%%%%
%%%%%%%%%%%%%%%%%%%%%%%%%%%%%%%%%%%%%%%%%%%%%%%%%%%%%%%%%%%%%%%%%%%%%%%%%%%%%%%%%%%%%%%%%%%%%%%%%%%%%%%%%%%%%

	Many algorithmic problems that are NP-complete for general graphs remain computationally intractable for bipartite graphs,
	which is the case, for instance, for  {\sc hamiltonian cycle} \cite{hamilton}, {\sc maximum induced matching} \cite{induced-matching}, {\sc alternating cycle-free matching} \cite{cycle-free},  
	{\sc balanced biclique} \cite{balanced}, {\sc maximum edge biclique} \cite{emb}, {\sc dominating set}, {\sc steiner tree} \cite{steiner}, {\sc independent domination} \cite{Damaschke},
	{\sc induced subgraph isomorphism} \cite{isi}. 
	
	The simple structure of chain graphs implies bounded clique-width and therefore po\-ly\-no\-mi\-al-time solvability of all these and many other problems. 
	However, in quasi-chain graphs the clique-width is unbounded and hence no solution comes for free in this class. Moreover, 
	{\sc induced subgraph isomorphism} remains intractable, as we show in Section~\ref{sec:isi} based on the relationship between quasi-chain graphs and permutations revealed in Theorem~\ref{thm:perm}.
	
	On the other hand, the structure of quasi-chain graphs revealed in Theorem~\ref{thm:decomposition} allows us
	to prove polynomial-time solvability of three problems in the above list, which we do in Section~\ref{sec:poly}.

	\subsection{NP-completeness of {\sc induced subgraph isomorphism} in quasi-chain graphs}
	\label{sec:isi}
	
	The {\sc induced subgraph isomorphism} problem can be stated as follows: given two graphs $H$ and $G$, decide whether $H$ is an induced subgraph of $G$ or not. 
	This problem is known to be NP-complete even when both graphs are bipartite permutation graphs \cite{isi}.
	A related problem on permutations is known as {\sc pattern matching}: given two permutations $\pi$ and $\rho$, it asks whether $\pi$ contains $\rho$ as a pattern.
	This problem is also NP-complete \cite{perm}. Together with Theorem~\ref{thm:perm} this immediately implies that {\it coloured} {\sc induced subgraph isomorphism} is NP-complete for quasi-chain graphs.
	Below we extend this conclusion to uncoloured graphs.
	
	\begin{theorem} 
		The {\sc induced subgraph isomorphism} problem is NP-complete for quasi-chain graphs. 
	\end{theorem}
	
	\begin{proof}
		Let $H$ and $G$ be two coloured connected quasi-chain graphs. The NP-completeness of {\sc pattern matching} together with Theorem~\ref{thm:perm} imply that determining whether 
		there is an embedding of $H$ into $G$ as an induced subgraph that respects the colours is an NP-complete problem. To reduce the problem to uncoloured graphs, we 
		modify the instance of the problem as follows.  
		
		Let $p$ be a natural number greater than the maximum vertex degree in $G$, and let $K_{1,p}$ be a star with the center $x$.
		We add this star to $G$, connect $x$ to all the black vertices of $G$ and denote the resulting graph by $G^*$. 
		Similarly, we add this star to $H$, connect $x$ to all the black vertices of $H$ and denote the resulting graph by $H^*$.
		Clearly, $G^*$ and $H^*$  are quasi-chain graphs.
		
		Now we ignore the colours and ask whether $G^*$ contains $H^*$ as an induced subgraph. If $G^*$ contains $H^*$,  
		then vertex $x$ in $H^*$ must map to vertex $x$ in $G^*$ (due to the degree condition), 
		and the vertices of $H$ in $H^*$ are mapped to the vertices of $G$ in $G^*$ in a colour-preserving way (due to the connectedness of $G$ and $H$). 
		Therefore, $G$ contains $H$ as a {\it coloured} induced subgraph if and only if $G^*$ contains $H^*$ as an induced subgraph. 
		Since $G^*$ and $H^*$ are quasi-chain graphs and these graphs can be obtained from $G$ and $H$ in polynomial time, 
		we conclude that {\sc induced subgraph isomorphism} is NP-complete for quasi-chain graphs. 
	\end{proof}

	\subsection{Polynomial-time algorithms for quasi-chain graphs}
	\label{sec:poly}
	
	In this section, we use Theorem~\ref{thm:decomposition} to prove polynomial-time solvability of the following problems in quasi-chain graphs: 
	{\sc balanced biclique}, {\sc maximum edge biclique}, and {\sc independent domination}. 
	We emphasize that Theorem~\ref{thm:decomposition}  not only provides a structural characterisation of quasi-chain graphs, 
	it also proves that a quasi-chain graph can be transformed into a chain graph by removing a matching and adding a matching in polynomial time, which is an important ingredient in all three solutions. 
	We start with an auxiliary lemma.
	
	\begin{lemma}\label{lem:opt}
		A quasi-chain graph $G$ with $n$ vertices contains a collection $\mathcal I$ of $O(n)$ subsets of vertices that can be found in polynomial time such that 
		every subset $I \in \mathcal I$ induces a graph of vertex degree at most 1, and every independent set in $G$ is contained in one of these subsets.
	\end{lemma}  
	
	\begin{proof}
		First, we observe that there are $O(n)$ inclusion-wise maximal independent sets in a chain graph, and that all of them can be found in polynomial time.
		
		Now let $G=Z\otimes H$ be a quasi-chain graph and let $S$ be an independent set in $G$.  Then in the graph $Z$, the vertices of $S$ either form an independent set, or
		induce some bottom edges, i.e., some edges of $E(H)\cap E(Z)$. Since bottom edges form a matching and $Z$ is $2P_2$-free, we conclude 
		that $S$ contains at most one bottom edge in the graph $Z$.
		
		If $S$ is an independent set in $Z$, then it is contained in a maximal independent set $I$ in $Z$. For each maximal independent set $I$ in the graph $Z$, 
		the vertices of $I$ induce in $G$ a subgraph $G[I]$ of vertex degree at most $1$, because all edges of $G[I]$ are top edges and therefore they form a matching.
		
		Assume now that $S$ contains an edge $a_ib_j$ in the graph $Z$. We denote the set of non-neighbours of $a_i$ in $G$ by $A_i$ and the set of non-neighbours of $b_j$ in $G$ by $B_j$, and let $I=A_i\cup B_j$.
		In particular, $S\subseteq I$. In $Z$, the vertices of $I$ induce a subgraph $Z[I]$ containing exactly one edge $a_ib_j$. Indeed, no edge $e\neq a_ib_j$ in $Z[I]$ can be incident to $a_i$ or $b_j$,
		because otherwise both $e$ and $a_ib_j$ are bottom edges, which is impossible, and if $e$ is not incident to $a_i$ and $b_j$, then $e$ and $a_ib_j$ create an induced $2P_2$ in $Z$, which is not possible either.
		Since $a_ib_j$ is the only edge in $Z[I]$ and this edge is not present in $G[I]$, we conclude that all edges of $G[I]$ are top edges and hence $G[I]$ is a graph of vertex degree at most one. 
		%Every such set $I$ is constructed from one of the bottom edges. 
		
		Putting everything together, our collection $\mathcal I$ consists of two types of sets: the maximal independent sets from $Z$, and the sets constructed as above from each of the bottom edges. This collection thus has $O(n)$ sets, and can be found in polynomial time as claimed.
	\end{proof}

	\subsubsection{Bicliques in quasi-chain bipartite graphs}

	A \emph{biclique} is a complete bipartite graph $K_{p,q}$ for some $p$ and $q$.
	In a bipartite graph, the problem of finding a biclique with the maximum number of vertices can be solved in polynomial time.
	However, the problem of finding a biclique with the maximum number of edges, known as the {\sc maximum edge biclique} problem,
	is NP-complete for bipartite graphs \cite{emb}. Additionally, the problem of finding a biclique $K_{p,p}$ with the maximum value of $p$, known as the {\sc balanced biclique} problem,
	is NP-complete for bipartite graphs \cite{balanced}. We show that both problems can be solved in polynomial time when restricted to quasi-chain graphs.

	\begin{theorem}
		The {\sc maximum edge biclique} and {\sc balanced biclique} problems can be solved in polynomial time for quasi-chain graphs.
	\end{theorem}
	
	\begin{proof}
		Let $G=(A,B,E)$ be a quasi-chain graph. A biclique in $G$ becomes an independent set in the bipartite complement $\widetilde{G}$ of $G$.
		Since $2P_3$ is self-complementary in the bipartite sense, we note that $\widetilde{G}$ is a quasi-chain graph too. 
		
		Let $\mathcal I$ be as in Lemma~\ref{lem:opt} for $\widetilde{G}$. Every independent set in $\widetilde{G}$ is contained in a maximal independent set, which in turn is contained in one of the subsets of $\mathcal I$.
		%each of which induces in $G'$ a graph of vertex degree at most 1, according to Lemma~\ref{lem:opt}. 
		In $G$, those subsets induce almost complete bipartite graphs, 
		i.e., graphs in which every vertex has at most one non-neighbour in the opposite part. 
		Therefore, to solve both problems for $G$, it suffices to solve them for this collection of $O(n)$ almost complete bipartite graphs.
		
		But those problems are both easy for almost complete bipartite graphs: suppose a graph is obtained from $K_{s,t}$ by deleting a matching of size $m\le s\le t$.
		It is not difficult to see that the number of edges in a maximum edge biclique in this graph equals $\max\limits_{0\leq i\leq m}(t-m+i)\cdot (s-i)$.
		As for the {\sc balanced biclique} problem, the optimal solution is given by $p = s$ if $t - s \geq m$, and by $\left\lfloor\frac{t - m + s}{2}\right\rfloor$ if $t - s < m$. 
	\end{proof}
	
	\subsubsection{Independent domination in quasi-chain graphs}
	
	The {\sc independent dominating set} problem asks to find in a graph $G$ an inclusion-wise maximal independent set of minimum cardinality. 
	This problem is NP-complete for general graphs and remains intractable in many restricted graph families. In particular, it is NP-complete  
	both for $2P_3$-free graphs \cite{zverovich} and for  bipartite graphs \cite{Damaschke}. 
	In the following theorem, we prove polynomial-time solvability of the problem for quasi-chain graphs.

	\begin{theorem}
		The {\sc independent dominating set} problem can be solved for  quasi-chain graphs in polynomial time.
	\end{theorem}
	
	\begin{proof}
		Let $G=(A,B,E)$ be a quasi-chain graph and $S$ an optimal solution to the problem in $G$, and let $\mathcal I$ be as in Lemma~\ref{lem:opt}. 
Note that $S$ is contained in at least one of the elements of $\mathcal I$. 
Moreover, crucially, for any $I \in \mathcal I$, all maximal independent sets in $G[I]$ have the same size. This suggests the following way of finding an optimal solution:
		
		\begin{enumerate}
			\item For each $I \in \mathcal I$, determine if $I$ contains an independent set that dominates $G$, and if yes, find such a set.
			
			\item Among the sets we found, pick one with minimum size.
		\end{enumerate}
		
We claim that this produces an optimal solution to the problem. Indeed, this procedure is guaranteed to produce a set $S$, 
since any optimal solution to the problem dominates $G$ and is contained in some $I \in \mathcal I$. 
Moreover, since all maximal independent sets in $G[I]$ have the same size (and $S$ dominates $G$, so it is maximal in both $G$ and $G[I]$), $S$ must be an optimal solution.
		
		It thus suffices to show that Step 1 can be done efficiently. To do this, let $I \in \mathcal I$. 
Let $I'\subseteq I$ be the subset of $I$ of vertices that have degree 1 in $G[I]$, and put $I'':=I-I'$.
 We note that any independent subset of $I$ dominating $G$ must contain all vertices of $I''$, and exactly one vertex from each edge of $G[I']$. 
Let $A''$ and $B''$ be the sets of vertices in $A$, respectively $B$ that have at least one neighbour in $I''$. 
We also denote $I'_A:=I'\cap A$ and $I'_B:=I'\cap B$, and  let $A'$ and $B'$ be the sets of vertices in $A - (A''\cup I'_A)$, respectively $B - (B''\cup I'_B)$ that have at least one neighbour in $I'$. 
		
		If $I$ does not dominate $G$, then no subset of $I$ dominates $G$; we may thus assume $I$ dominates $G$, that is, $A-I = A' \cup A''$ and $B-I = B' \cup B''$. 
Since $G$ is $2P_3$-free, the graphs $G[I'_A\cup B']$ and $G[I'_B\cup A']$ are $2P_2$-free, i.e., chain graphs. 
		It follows that $I'_A$ and $I'_B$ each have vertices that dominate $B'$ and $A'$ respectively. 
If there exists such a pair $x \in I'_A$ and $y \in I'_B$ that is {\it non-adjacent}, then we are done: 
we pick $x$ and $y$ in their respective edges, and arbitrarily choose vertices from each other edge of $I'$ to complete our  independent dominating set. 
Otherwise, the unique vertices $x \in I'_A$ and $y \in I'_B$ that dominate $B'$ and $A'$ respectively belong to the same edge of $I'$. 
In this case, no independent set of $I$ dominates $G$, since vertices $A'$ and $B'$ have no neighbours in $I''$ by construction, 
and (using $2P_2$-freeness) $I'_A - \{x\}$ does not dominate $A'$, and $I'_B - \{y\}$ does not dominate $B'$. This proves the theorem.
\end{proof}

\section{Conclusion}

In this paper, we proposed a structural characterization for the class of $2P_3$-free bipartite graphs and derived a number of interesting conclusions 
from this characterization. Still, many questions remain unanswered. In particular, it would be interesting to find a boundary separating well-quasi-ordered 
subclasses of quasi-chain graphs from those that contain infinite antichains with respect to the induced subgraph relation. Also, complexity of several 
important algorithmic problems in the class of quasi-chain graphs remain unknown. One more important direction of research is analyzing the extension of 
quasi-chain graphs, where a ``one-sided'' copy of a $2P_3$ is forbidden, i.e. the class of coloured bipartite graphs that do not contain an induced copy of
$2P_3$ with, say, white centres. In particular, does this extension admit an implicit representation?  

%%%%%%%%%%%%%%%%%%%%%%%%%%%

\end{document}